\documentclass[a4paper, 12pt]{amsart}
\usepackage{amssymb}
\usepackage{amsthm}
\usepackage{amsmath}
\usepackage{tikz}
\usepackage{graphicx}
\usepackage{caption}
\usepackage{subcaption}
\usepackage[utf8]{inputenc}
\usepackage[english]{babel}
\usepackage[shortlabels]{enumitem}
\usepackage{comment}
\usepackage{faktor}

\newcommand{\Filip}[1]{{\sf $\diamondsuit$ [#1]}}
\newcommand{\bek}[1]{{\sf $\clubsuit$ [#1]}}

\DeclareMathOperator{\HS}{\mathrm{HS}}

\newtheorem{theorem}{Theorem}[section]

\newtheorem{lemma}[theorem]{Lemma}

\newtheorem{proposition}[theorem]{Proposition}

\theoremstyle{definition}
\newtheorem{definition}[theorem]{Definition}
\newtheorem{example}[theorem]{Example}
\newtheorem{remark}[theorem]{Remark}

\title[The SLP for some monomial almost CI's]{The strong Lefschetz property for some monomial almost complete intersections}

\author{Bek Chase and Filip Jonsson Kling}

\address{Department of Mathematics\\ University of Central Arkansas\\ 201 Donaghey Ave, Conway, AR 72035, USA}
\email{rchase@uca.edu}

\address{Department of Mathematics\\ Stockholm University\\ SE-106 91 Stockholm, Sweden}
\email{filip.jonsson.kling@math.su.se}

\begin{document}

\begin{abstract}
Motivated by the foundational result that a monomial complete intersection has the strong Lefschetz property (SLP) in characteristic zero, it is natural to ask when monomial almost complete intersections have the SLP. In this paper, using the Hilbert series as a central tool, we investigate the strong Lefschetz property for certain monomial almost complete intersections, those with the non-pure-power generator having support in two variables, and those with symmetric Hilbert series. In the former case, we give a complete classification for when the SLP holds, and in the latter case, we prove that such algebras always have the SLP.
\end{abstract}
\maketitle

\section{Introduction}
Consider a polynomial ring $R=k[x_1,\dots, x_n]$ over a field $k$ of characteristic zero and let $I$ be some Artinian ideal defining a $k$-algebra $A=R/I$. An important open problem is to determine for which families of $k$-algebras $A$ there exists a linear form $\ell\in A_1$ such that all multiplication maps
\[
\cdot \ell^t : A_i \to A_{i+t}
\]
have maximal rank (i.e., each map is injective or surjective). If the maps have full rank for $t=1$, then we say that $A$ has the \emph{weak Lefschetz property} (WLP), and if it holds for all $t\geq 1$, then we say $A$ has the \emph{strong Lefschetz property} (SLP). By an abuse of notation we will sometimes simply say the ideal $I$ satisfies the WLP or SLP when $A$ does. 

The study of Lefschetz properties is an active area of research in commutative algebra with connections to several other areas of mathematics such as algebraic geometry, topology and combinatorics. The first major theorem in the area was proven independently by Stanley \cite{Stanley} and Watanabe \cite{Watanabe}, showing that when $A$ is defined by an (Artinian) monomial complete intersection $I=(x_1^{a_1},\dots, x_n^{a_n})$, then $A$ has the SLP. 

A natural next step is then to consider Artinian algebras defined by monomial ideals having one additional generator, thus minimally generated by $n+1$ monomials instead. These ideals are called \emph{monomial almost complete intersections} (MACIs). Such algebras can fail not only the SLP, but also the WLP. As a first example, we have the Togliatti system $I=(x_1^3, x_2^3, x_3^3, x_1x_2x_3)$ which fails the WLP. Remarkably, Brenner and Kaid \cite{Brenner-Kaid} proved that this is the only homogeneous ideal generated by four cubics $x_1^3,x_2^3,x_3^3$ and some $f$ of degree $3$ in three variables which fails the WLP. Later, both families of monomial almost complete intersections failing WLP \cite{MMRN-Almost-CI}, having the SLP \cite{Gluing_SashaSamuelLisa}, and families where having the WLP or not is a question of non-trivial combinatorics \cite{Cook-Nagel-Almost-CI}, have been established. 

Our goal here is to study the strong Lefschetz property for two particular families of monomial almost complete intersections.\newpage

\begin{itemize}
    \item[(i)] $I=(x_1^{a_1},\dots, x_n^{a_n},m)$, where $m = x_i^{\alpha}x_j^{\beta}$ has support in two variables.
    \item[(ii)] $I=(x_1^{a_1},\dots, x_n^{a_n},m)$, where $m \neq 0$ is any monomial such that $R/I$ has a symmetric Hilbert series.
\end{itemize}
In both cases we give a complete classification for the SLP. In particular, in the first case, we give the specific criteria for which $\alpha$ and $\beta$ the ideal has the SLP (Theorem \ref{thm:SLP_supp_two}), while in the second case, we prove that it always has the SLP (Theorem \ref{thm:slp_symmetric_aci}). 

This article is laid out as follows. In Section \ref{sect:preliminaries}, we recall necessary facts which will be used throughout. Section \ref{sect:small_support} focuses on the first case $(i)$ of monomial almost complete intersections from above. We prove that such an algebra $A$ satisfies the strong Lefschetz property whenever (assuming $m = x_1^{\alpha}x_2^{\beta}$) the Hilbert series of $A/(x_3,\dots,x_n)$ satisfies certain conditions corresponding to numerical conditions on the exponents of the generators by a thorough study of the codimension two case. In the last section, we prove that any monomial almost complete intersection with a symmetric Hilbert series has the strong Lefschetz property, using the central simple modules of Harima and Watanabe \cite{central-simple, nonstd} as a key tool. We end with a remark discussing the potential of our results to also provide knowledge of the Lefschetz properties for almost complete intersections which are not necessarily monomial. 

\section{Preliminaries}\label{sect:preliminaries}
Throughout this paper, we will be working with the polynomial ring $R=k[x_1,\dots, x_n]$ over a field $k$ of characteristic zero. Since our main focus are monomial almost complete intersections, we begin by giving a definition for these algebras. 

\begin{definition}
A \emph{monomial almost complete intersection} in $k[x_1,\dots, x_n]$ is an Artinian monomial ideal with $n+1$ minimal generators.
\end{definition}

Using that pure powers of all variables must appear as minimal generators of an Artinian monomial ideal, we get that a monomial almost complete intersection may equivalently be defined as an ideal of the form 
\[
I=(x_1^{a_1},\dots, x_n^{a_n},m)
\]
where $m$ is some monomial not contained in $(x_1^{a_1},\dots, x_n^{a_n})$. The main tool used to understand these ideals is the Hilbert series. 


\begin{definition}
Let $A = \bigoplus_{i=0}^c A_i$ be a standard graded Artinian $k$-algebra where $k$ is a field of characteristic zero. If $M$ is a finite graded Artinian $A$-module, then we can write $M = \bigoplus_{i=p}^q M_i$, where $M_p \neq 0$ and $M_q \neq 0$. Here $q$ is the socle degree of $M$, i.e., $q$ is the largest degree of an element in $M$ which is annihilated by the ideal $(x_1,\dots,x_n)$. The Hilbert function of $M$ is the map $i \mapsto h_i := \text{dim}_k M_i$ and the Hilbert series of $M$, which when $M$ is Artinian is a Laurent polynomial, is given by
\[
\HS_M(t) = \sum_i h_i t^i.
\]
If the Hilbert series of $M$ is symmetric, that is, if $\text{dim}_k M_{p+i} = \text{dim}_k M_{q-i}$ for all $i$, then we call $(p+q)/2$ the \textit{reflecting degree} of $\HS_M(t)$, or just the reflecting degree of $M$ for convenience. If $M$ and $M'$ are two finite $A$-modules, then we say the reflecting degrees $r$ of $M$ and $r'$ of $M'$ \textit{coincide} if $r = r'$ or $|r-r'| = \frac{1}{2}$.
\end{definition}
We will sometimes illustrate Hilbert series pictorially as in Figure \ref{fig:Basic_Hilbert}. In this example, the corresponding Hilbert function is nonzero for $i$ in the interval from $0$ up to and including $c$. In particular, the values of the Hilbert function are increasing from $0$ to $a$, constant between $a$ and $b$, and then decreasing from $b$ to $c$. 
\begin{figure}[h]
\centering
\begin{tikzpicture}[scale=0.9]
\draw (-2,0) -- (0,2);
\draw (0,2) -- (3,2);
\draw (3,2) -- (5,0);
\node[below] at (-2,0) {\tiny{$0$}};
\node[below] at (5,0) {\tiny{$c$}};
\node[above] at (0,2) {\tiny{$a$}};
\node[above] at (3,2) {\tiny{$b$}};
\end{tikzpicture}
\caption{}
\label{fig:Basic_Hilbert}
\end{figure}

One particular class of Hilbert series that will be important to us is the following.

\begin{definition}
Let $\sum_i h_i t^i$ be the Hilbert series of an algebra $A$ with socle degree $D$. We say that the Hilbert series of $A$, or sometimes just $A$, is \emph{almost centered} if it satisfies
\begin{align*}
h_{i-1} \leq h_{D-i} \leq h_i \text{ for all } &0 \leq i \leq \lfloor \frac{D}{2} \rfloor, \text{ or}\\
h_{D-i+1} \leq h_{i} \leq h_{D-i} \text{ for all } &0 \leq i \leq \lfloor \frac{D}{2} \rfloor.
\end{align*}
\end{definition}

\begin{remark}
\label{rem:ac_version2}
Sometimes it can be useful to utilize an equivalent definition of almost centered algebras. By \cite[Corollary 5.6]{full_rank_recursion}, $A$ is almost centered if and only if both of the following properties are satisfied.
\begin{itemize}
\item If $h_i<h_j$ for some $i<j$, then $h_{i-s}\leq h_{j+s}$ for any $s \geq 0$.
\item If $h_i>h_j$ for some $i<j$, then $h_{i-s}\geq h_{j+s}$ for any $s\geq 0$.
\end{itemize}
\end{remark}

Almost centered Hilbert series were first introduced by Lindsey in order to prove the following theorem.

\begin{theorem}[{\cite[Theorem 3.10]{Lindsey}}]
Let $A$ be a standard graded Artinian $k$-algebra with the SLP over a field $k$ of characteristic zero with Lefschetz element $\ell$. Then $A\otimes_k k[x]/(x^d)$ has the SLP for all $d\geq 1$ if and only if $A$ is almost centered.
\end{theorem}

In other words, you can always add a power of a new variable to an almost centered algebra and that will preserve the SLP if the original algebra has the SLP.

An important construction which we utilize in Section \ref{sect:symmetric_hs} is Harima and Watanabe's central simple modules, which we define below (c.f.
\cite{central-simple, nonstd, commutator}). This gives us another way, aside from tensor products, to simplify the problem of proving an algebra has the SLP by breaking it down into more manageable pieces. 

\begin{definition}
Let $A$ be a standard graded Artinian $k$-algebra and $\ell \in A_1$ a linear form with $r$ the smallest positive integer for which $\ell^r = 0$. Then we have the sequence 
\[
A = (0:\ell^r)+(\ell)  \supset (0:\ell^{r-1}) +(\ell) \supset \dots \supset (0:\ell)+(\ell) \supset (0:\ell^0)+(\ell) = (\ell).
\]
The $i$th \textit{central simple module} $V_{i,\ell}$ of $A$ with respect to the linear form $\ell$ is defined as the $i$th nonzero successive quotient of this sequence, so it is of the form
\[
V_{i,\ell} = \frac{(0:\ell^{f_i})+(\ell)}{(0:\ell^{f_{i}-1})+(\ell)}
\]   
where $f_1 = r$ and $f_i > f_{i+1} \geq 1$ for all $i >1$. Note that $V_{1,\ell} = A/{(0:\ell^{r-1})+(\ell)}$. We also define $\widetilde{V}_{i,\ell} = V_{i,\ell} \otimes k[t]/(t^{f_i})$ and $\widetilde{V} = \bigoplus_i \widetilde{V}_{i,\ell}$. Then, viewing $\widetilde{V}_{i,\ell}$ and $K[t]/(t^{f_i})$ as $k$-vector spaces, we have \[\HS_{\widetilde{V}_{i,\ell}}(t) = \HS_{V_{i,\ell}}(t)(1+t+\dots + t^{f_i -1})\]
which is symmetric if and only if $\HS_{V_{i,\ell}}(t)$ is symmetric \cite[Remark 2.2(1)]{nonstd}.
\end{definition}

We also recall the theorem of Harima and Watanabe giving a characterization of the strong Lefschetz property in terms of these central simple modules. 

\begin{theorem}[{\cite[Theorem 5.2]{nonstd}}]\label{hw-csm-thm}
Let $A$ be a graded Artinian $k$-algebra with a symmetric Hilbert series. Then the following are equivalent. 
\begin{enumerate}
    \item $A$ has the SLP. 
    \item There exists a linear form $\ell \in A_1$ such that for all $i$, $V_{i,\ell}$ has the SLP, a symmetric Hilbert series, and the reflecting degree of $\widetilde{V}_{i,\ell}$ coincides with that of $A$.
\end{enumerate} 
\end{theorem}

An important fact about central simple modules which we will use later is that $\HS_A(t) = \sum_i \HS_{\widetilde{V}_i}(t)$ (per \cite[Remark 4.1]{nonstd}). 

\section{Monomials with small support}\label{sect:small_support}
Since a full classification of the SLP for all monomial almost complete intersections is a very difficult task, in this section we will begin by assuming that the additional monomial $m$ has as small support as possible. If $I$ should be a monomial almost complete intersection, this forces $m$ to be of the form $x_i^{\alpha}x_{j}^{\beta}$ for some $i,j$ and $\alpha, \beta \geq 1$. By changing the labels of our variables, we may assume that $m=x_1^{\alpha}x_2^{\beta}$. The goal for this section is then to prove the following theorem, characterizing when such ideals have the SLP.

\begin{theorem}
\label{thm:SLP_supp_two}
Let $A=k[x_1,\dots, x_n]/(x_1^{a_1},\dots, x_n^{a_n}, x_1^{\alpha}x_2^{\beta})$ be a monomial almost complete intersection over a field of characteristic zero. Without loss of generality, assume that $a_1 + \beta \leq a_2 + \alpha$. Then $A$ has the SLP if and only if one of the following conditions is satisfied.
\begin{itemize}
\item $n=2$,
\item $n=3$ and $a_3\leq 2$,
\item $k[x_1,x_2]/(x_1^{a_1},x_2^{a_2},x_1^{\alpha}x_2^{\beta})$ is almost centered, or
\item $a_2< a_1 + \beta + 2$ and
\begin{itemize}
\item[$\star$] $a_1=\alpha+1$, or
\item[$\star$] $\beta=1$, or
\item[$\star$] $a_2\geq a_1 + \beta - 1$.
\end{itemize}
\end{itemize}
\end{theorem}

Note that the assumption $a_1 + \beta \leq a_2 + \alpha$ can always be achieved by relabeling the variables $x_1$ and $x_2$ if necessary, exchanging their roles. Moreover, the last two criteria in Theorem~\ref{thm:SLP_supp_two} are equivalent; the last gives explicit conditions for when $k[x_1,x_2]/(x_1^{a_1}, x_2^{a_2}, x_1^{\alpha}x_2^{\beta})$ is almost centered, as we will see in Lemma \ref{lem:Two_var_Hilbert}. 

\begin{example}
The algebra $\mathbb{Q}[x_1,x_2,x_3,x_4]/(x_1^4, x_2^6, x_3^{a_3}, x_4^{a_4}, x_1^{2}x_2^3)$ has the SLP for all $a_3,a_4\geq 2$. Indeed, in this case $a_1=4$, $a_2=6$, $\alpha=2$ and $\beta=3$ satisfies $a_1+\beta\leq a_2+ \alpha$, $a_2<a_1+\beta+2$ and $a_2\geq a_1+\beta -1$. However, if we increase the value of $\beta$ by one so that $\beta=4$, then the algebra will fail the SLP since none of the conditions in Theorem \ref{thm:SLP_supp_two} is fulfilled. 
\end{example}

One well-known result that we will use repeatedly is the following, essentially coming from the additivity of Hilbert series on short exact sequences. For a proof see e.g. \cite[Lemma 1]{Gluing_SashaSamuelLisa}.

\begin{lemma}\label{lem:Hilbert_add}
Let $K$ be a ideal of $R$ and $m$ a monomial. Then, if $I=K+(m)$ and $J=K:(m)$, we get that
\[
\HS_{R/K}(t) = \HS_{R/I}(t) + t^{\deg(m)}\HS_{R/J}(t).
\]
\end{lemma}

To prove Theorem \ref{thm:SLP_supp_two}, the main part of the work lies in understanding the Hilbert series of a monomial almost complete intersection in two variables. 

\begin{lemma}
\label{lem:Two_var_Hilbert}
Assume that $a+\beta\leq b+\alpha$. Then $\HS_{k[x,y]/(x^a, y^b, x^\alpha y^\beta)}(t)$ has the following properties.
\begin{itemize}
\item It is unimodal.
\item It increases by one in each degree until it reaches its maximum. 
\item After it reaches its maximum, it weakly decreases in steps of size zero, one or two.
\item The maximum is obtained in degree $\min\{a,\beta+\alpha\}-1$ (and possibly in other degrees as well). 
\item The socle degree is $b+\alpha-2$.
\item It is symmetric if and only if $a+\beta=b$.
\item It is \emph{not} almost centered if and only if one of the following criteria is satisfied.
\begin{itemize}
\item[$\star$] $b\geq a+\beta+2$
\item[$\star$] $a-\alpha\geq 2, \beta\geq 2$ and $b\leq a+\beta-2$.
\end{itemize}
\end{itemize}
\end{lemma}

\begin{proof}
Let $R=k[x,y]$ and $K=(x^a, y^b, x^\alpha y^\beta)$. To understand the Hilbert series of $R/K$, we will use Lemma \ref{lem:Hilbert_add}. For $m=x^\alpha$, $I=K+(m)=(x^{\alpha},y^b)$ and $J=K:(m)=(x^{a-\alpha},y^\beta)$, it gives that that
\[
\HS_{R/K}(t) = \HS_{R/I}(t) + t^\alpha\HS_{R/J}(t).
\]
Now, $I$ and $J$ are both monomial complete intersections in two variables, so we understand their Hilbert series. We have that $\HS(R/I)$ is strictly increasing by one in the interval of integers $[0, \min\{b,\alpha\}-1] \subset \mathbb{Z}$, constant in the interval $[\min\{b,\alpha\}-1,\max\{b,\alpha\}-1]$, and strictly decreasing by one in the interval $[\max\{b,\alpha\}-1, b+\alpha-2]$. Similarly for $t^\alpha\HS_{R/J}(t)$, it is strictly increasing by one in the interval $[\alpha, \min\{a,\beta+\alpha\}-1]$, constant in the interval $[\min\{a,\beta+\alpha\}-1,\max\{a,\beta+\alpha\}-1]$, and strictly decreasing by one in the interval $[\max\{a,\beta+\alpha\}-1, a+\beta-2]$. See Figure \ref{fig:General_I} and Figure \ref{fig:General_J} respectively. 

\begin{figure}[h]
\begin{minipage}{0.5\textwidth}
\begin{tikzpicture}[scale=0.9]
\draw (-2,0) -- (0,2);
\draw (0,2) -- (3,2);
\draw (3,2) -- (5,0);
\node[below] at (-2,0) {\tiny{$0$}};
\node[below] at (5,0) {\tiny{$b+\alpha-2$}};
\node[above] at (0,2) {\tiny{$\min\{b,\alpha\}-1$}};
\node[above] at (3,2) {\tiny{$\max\{b,\alpha\}-1$}};
\end{tikzpicture}
\caption{$\mathrm{HS}_{R/I}(t)$}
\label{fig:General_I}
\end{minipage}~
\begin{minipage}{0.5\textwidth}
\begin{tikzpicture}[scale=0.9]
\draw[blue] (-2,0) -- (0,2);
\draw[blue] (0,2) -- (3,2);
\draw[blue] (3,2) -- (5,0);
\node[below] at (-2,0) {\tiny{$\alpha$}};
\node[below] at (5,0) {\tiny{$a+\beta-2$}};
\node[above] at (0,2) {\tiny{$\min\{a,\beta+\alpha\}-1$}};
\node[above] at (3,2) {\tiny{$\max\{a,\beta+\alpha\}-1$}};
\end{tikzpicture}
\caption{$t^\alpha\: \mathrm{HS}_{R/J}(t)$}
\label{fig:General_J}
\end{minipage}
\end{figure}

The possible shapes for $\HS_{R/K}(t)$ are then collected in Figure \ref{fig:All_on_top} and Figure \ref{figs:Slanted} where we let blue denote the part coming from $t^\alpha\: \mathrm{HS}_{R/J}(t)$ while keeping the part of $\mathrm{HS}_{R/I}(t)$ below it with a dashed line. These are constructed by combining the facts that $a+\beta-2\leq b+\alpha-2$, giving that $t^\alpha\: \mathrm{HS}_{R/J}(t)$ must end before $\mathrm{HS}_{R/I}(t)$, together with that $t^\alpha\: \mathrm{HS}_{R/J}(t)$ starts directly at or directly after the flat top of $\mathrm{HS}_{R/I}(t)$. These restrictions then give these finitely many possible shapes. From these pictures, most of the claimed properties of $\HS_{R/K}(t)$ are immediate such as unimodality and that the socle degree is $b+\alpha-2$. The amount of increase and decrease and the location of the maximum also follows where we note that a steeper slope in the figures indicates a decrease by two in each degree. Moreover, we see that the only case when the series is symmetric happens in Figure \ref{fig:All_on_top} when $a+\beta-1=b-1$. 

It remains to determine when $\HS_{R/K}(t)$ is almost centered. Let us write 
\[
\HS_{R/K}(t) = \sum_{i=0}^{b+\alpha-2}h_i t^i.
\]
Assume first that $a+\beta-1\leq\max\{b,\alpha\}$. Since $\beta\geq 1$ and $\alpha<a$, we then know that $a+\beta-2\geq a-1>\alpha-1$, so $\max\{b,\alpha\}-1=b-1$ and we are in the case indicted in Figure \ref{fig:All_on_top}.  If $a+\beta-1=b-1$, then the series is unimodal and symmetric and thus almost centered. If $a+\beta-1=b-2$, then the series will again be almost centered. Indeed, we have $h_{b+\alpha-2}\leq h_0 \leq h_{b+\alpha-3}\leq h_1 \leq \dots$ where we must take this ordering of the coefficients of the Hilbert series and not the other to counteract the fact that the series is not symmetric anymore. One can write down the explicit relationship between the coefficients of the Hilbert series prove this, but since all info is collected in Figure \ref{fig:All_on_top}, it does not add any extra insight and is unnecessary. The case when $a+\beta-1=b$, though pictorially closer to Figure \ref{fig:First_slanted}, follows by the same logic, establishing $h_0\leq h_{b+\alpha-2}\leq h_1 \leq h_{b+\alpha-1}\leq \dots$ instead. In the case when $a+\beta-1\leq b-3$, then it can not be almost centered since then we have that $h_\alpha>h_{b-3}$ and $h_{\alpha-2}<h_{b-1}$, contradicting the almost centered property from Remark \ref{rem:ac_version2}. 

\begin{figure}[h]
\centering
\begin{tikzpicture}
\draw (-1.5,0.5) -- (0,2);
\draw[dashed] (0,2) -- (3,2);
\draw (3,2) -- (4,2);
\draw (4,2) -- (5.5,0.5);
\draw[blue] (0,2) -- (1,3);
\draw[blue] (1,3) -- (2,3);
\draw[blue] (2,3) -- (3,2);
\node[below] at (-1.5,0.5) {\tiny{$0$}};
\node[below] at (5.5,0.5) {\tiny{$b+\alpha-2$}};
\node[above left] at (0,2) {\tiny{$\alpha-1$}};
\node[above] at (4,2) {\tiny{$b-1$}};
\node[below] at (3,2) {\tiny{$a+\beta-1$}};
\end{tikzpicture}
\caption{$\HS_{R/K}(t)$ when $a+\beta-2\leq b-1$}
\label{fig:All_on_top}
\end{figure}
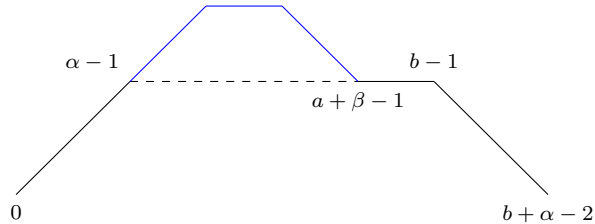

Suppose now that $\max\{b,\alpha\}+1\leq a+\beta-1$, so that we are in one of the remaining cases in Figure \ref{figs:Slanted}. Since we always have that $a+\beta-1\geq a \geq \alpha+1$, we may equivalently assume that $b+1\leq a+\beta-1$. If $a-\alpha=1$ or $\beta=1$, then $t^\alpha\: \mathrm{HS}_{R/J}(t)$ has only ones as its non-zero coefficients. Adding that as described to $\HS_{R/I}(t)$ then gives that $\HS_{R/K}(t)$ is almost centered since the inequalities $h_0\leq h_{b+\alpha-2}\leq h_1 \leq h_{b+\alpha-3}\leq \dots$ will hold. This follows since $\HS_{R/I}(t)$ is symmetric and unimodal, and increasing the coefficients of some consecutive numbers in such as sequence by one can then not ruin the almost centered property. 

Hence we may assume that $a-\alpha\geq 2$ and $\beta\geq 2$ while still $b+1\leq a+\beta-1$ for our final case. But if $a+\beta-1\geq b+1$, then it cannot be almost centered as the assumptions $a-\alpha\geq 2$ and $\beta\geq 2$ ensures that the coefficient of $t^{a+\beta-3}$ in $t^\alpha\: \mathrm{HS}_{R/J}(t)$ is $2$, giving that $h_{a+\beta-3}=h_{b+\alpha-a-\beta+1}+2 = h_{b+\alpha-a-\beta+2} +1$, and we need $h_{a+\beta-3}\leq h_{b+\alpha-a-\beta+1}$ or $h_{a+\beta-3}\leq h_{b+\alpha-a-\beta+2}$ to be almost centered. Hence it fails to be almost centered exactly in the desired cases.
\end{proof}



\begin{figure}[h]
\begin{subfigure}{0.5\textwidth}
\centering
\begin{tikzpicture}[scale=0.85]
\draw (-1.5,-0.5) -- (1,2);
\draw (1,2) -- (3,2);
\draw[dashed] (3,2) -- (5,0);
\draw (5,0) -- (5.5,-0.5);
\draw[blue] (3,2) -- (3.75,2);
\draw[blue] (3.75,2) -- (4.25,1.5);
\draw[blue] (4.25,1.5) -- (5,0);
\node[below] at (-1.5,-0.5) {\tiny{$0$}};
\node[below] at (5.5,-0.5) {\tiny{$b+\alpha-2$}};
\node[above] at (1,2) {\tiny{$b-1$}};
\node[above] at (3,2) {\tiny{$\alpha-1$}};
\node[above right] at (5,0) {\tiny{$a+\beta-1$}};
\end{tikzpicture}
\caption{$b\leq \alpha$}
\label{fig:All_on_right}
\end{subfigure}~
\begin{subfigure}{0.5\textwidth}
\centering
\begin{tikzpicture}[scale=0.85]
\draw (-1.5,0.5) -- (0,2);
\draw[dashed] (0,2) -- (4,2);
\draw[dashed] (4,2) -- (5,1);
\draw (5,1) -- (5.5,0.5);
\draw[blue] (0,2) -- (2,4);
\draw[blue] (2,4) -- (3,4);
\draw[blue] (3,4) -- (4,3);
\draw[blue] (4,3) -- (5,1);
\node[below] at (-1.5,0.5) {\tiny{$0$}};
\node[below] at (5.5,0.5) {\tiny{$b+\alpha-2$}};
\node[above left] at (0,2) {\tiny{$\alpha-1$}};
\node[above left] at (4,2) {\tiny{$b-1$}};
\node[above right] at (5,1) {\tiny{$a+\beta-1$}};
\end{tikzpicture}
\caption{$\alpha\leq b$}
\label{fig:First_slanted}
\end{subfigure}
\begin{subfigure}{0.5\textwidth}
\centering
\begin{tikzpicture}[scale=0.85]
\draw (-2.5,-0.5) -- (0,2);
\draw[dashed] (0,2) -- (3,2);
\draw[dashed] (3,2) -- (5,0);
\draw (5,0) -- (5.5,-0.5);
\draw[blue] (0,2) -- (1,3);
\draw[blue] (1,3) -- (3,3);
\draw[blue] (3,3) -- (4,2);
\draw[blue] (4,2) -- (5,0);
\node[below] at (-2.5,-0.5) {\tiny{$0$}};
\node[below] at (5.5,-0.5) {\tiny{$b+\alpha-2$}};
\node[above left] at (0,2) {\tiny{$\alpha-1$}};
\node[above] at (3,2) {\tiny{$b-1$}};
\node[above right] at (5,0) {\tiny{$a+\beta-1$}};
\end{tikzpicture}
\caption{$\alpha\leq b$}
\label{fig:Second_slanted}
\end{subfigure}~
\begin{subfigure}{0.5\textwidth}
\centering
\begin{tikzpicture}[scale=0.85]
\draw (-2.5,-0.5) -- (0,2);
\draw[dashed] (0,2) -- (0.5,2);
\draw[dashed] (0.5,2) -- (2.5,0);
\draw (2.5,0) -- (3,-0.5);
\draw[blue] (0,2) -- (0.5,2.5);
\draw[blue] (0.5,2.5) -- (1,2.5);
\draw[blue] (1,2.5) -- (1.5,2);
\draw[blue] (1.5,2) -- (2.5,0);
\node[below] at (-2.5,-0.5) {\tiny{$0$}};
\node[below] at (3,-0.5) {\tiny{$b+\alpha-2$}};
\node[left] at (0,2) {\tiny{$\alpha-1$}};
\node[right] at (0.5,2) {\tiny{$b-1$}};
\node[above right] at (2.5,0) {\tiny{$a+\beta-1$}};
\end{tikzpicture}
\caption{$\alpha\leq b$}
\label{fig:Final_slanted}
\end{subfigure}
\caption{$\HS_{R/K}(t)$ when $b-1\leq a+\beta-2$}
\label{figs:Slanted}
\end{figure}
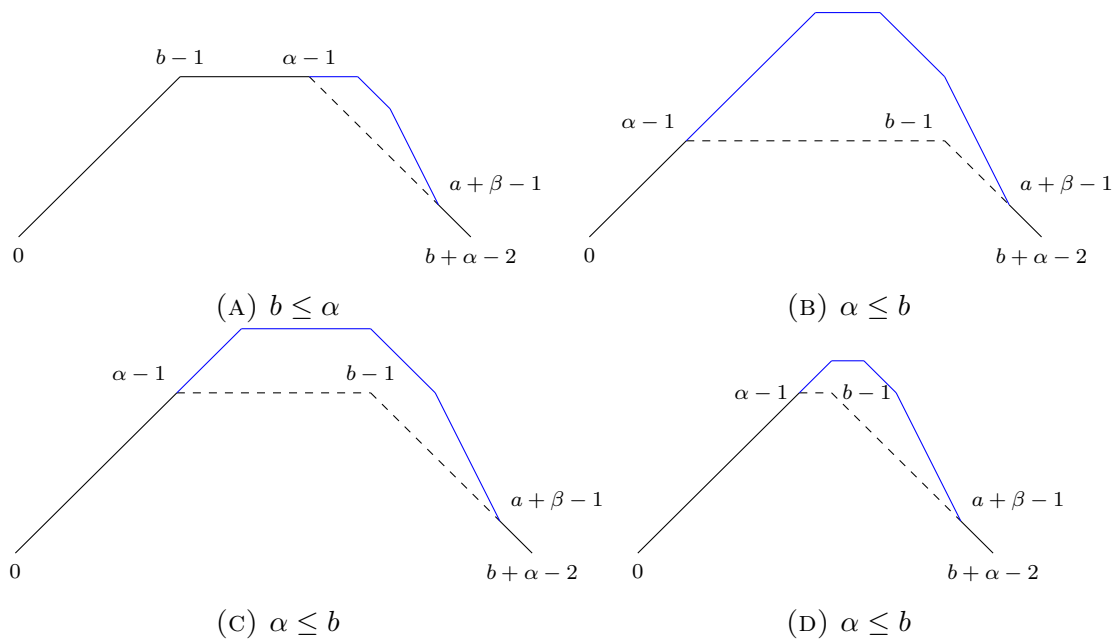

For some of the forthcoming proofs, we will need a result in the same spirit as the Clebsch-Gordan Theorem (see e.g. \cite[Proposition 3.66]{Lefschetz_book}), which gives a formula for the Jordan type of a linear form in a tensor product of two algebras in terms the Jordan types of a linear form in each algebra separately. However, the following theorem is more directly applicable to our purposes.  



\begin{theorem}[{\cite[Theorem 3.6]{full_rank_recursion}}]\label{thm:full_rank_recursion}
Let $A=B\otimes_k k[x]/(x^d)$ for some $d\geq 1$ where $B$ is an Artinian $k$-algebra and $k$ is a field of characteristic zero. For $i,t\geq 0$ and $\overline{\ell}\in B_1$ a general linear form, we have that the multiplication $\cdot \ell^t:A_i \to A_{i+t}$ for a general linear form $\ell = \overline{\ell} + x\in A_1$ has full rank if and only if all the maps
\[
\cdot \overline{\ell}^{2q + t -(d-1)}:B_{i-q} \to B_{i+q+t-(d-1)}
\]
where $q$ goes from $\max\{0,d-t\}$ to $d-1$ has full rank for the same reason. 
\end{theorem}

Having full rank for the same reason means here that either all of the maps are injective, or all of the maps are surjective. With all of this info on the Hilbert series in the two variable case, we can now start saying things about Lefschetz properties for $k[x,y,z]/(x^a, y^b, z^c, x^\alpha y^\beta)$.

\begin{lemma}
\label{lem:3_var_has_SLP}
The algebra $A=k[x,y,z]/(x^a, y^b, z^c, x^\alpha y^\beta)$ has the SLP if $c\leq 2$ or $B=k[x,y]/(x^a, y^b, x^\alpha y^\beta)$ is almost centered.
\end{lemma}
\begin{proof}
First, recall that any Artinian algebra in two variables over a field of characteristic zero has the SLP \cite[Proposition 4.4]{Codim_2}. Hence $B$ always has the SLP and so does $A$ when $c=1$. When $B$ is almost centered, we get that $A=B\otimes_k k[z]/(z^c)$ has the SLP for any $c\geq 1$ by Lindsey's theorem \cite[Theorem 3.10]{Lindsey}. We are now left with the case $c=2$, where we will use Theorem \ref{thm:full_rank_recursion}. It gives that $\cdot \ell^t:A_i \to A_{i+t}$ has full rank if $\cdot \overline{\ell}^{t-1}:B_{i}\to B_{i+t-1}$ and $\cdot \overline{\ell}^{t+1}:B_{i-1}\to B_{i+t}$ have full rank for the same reason. Since $B$ has the SLP, we only need to check that they do not have full rank for different reasons. Let $h_i=\dim_k(B_i)$. If $h_i=h_{i+t-1}$, then the first map is bijective and it does not matter if the second map is injective or surjective. If $h_i<h_{i+t-1}$, then using Lemma \ref{lem:Two_var_Hilbert}, $h_i$ must be in the increasing part of the Hilbert series, so $h_{i-1}=h_i-1$. Next, we know from the same lemma that $h_{i+t}\geq h_{i+t-1}-2$ since the Hilbert series can decrease by steps of at most two. Hence we get that $$h_{i-1}=h_i -1 \leq h_{i+t-1}-2\leq h_{i+t}$$ and both maps are injective. Finally, if $h_i>h_{i+t-1}$, then as the series is unimodal, we know that $h_{i+t-1}\geq h_{i+t}$. Moreover, since either $h_{i-1}\geq h_i$ or $h_{i-1}=h_i-1$, it follows that $$h_{i-1}\geq h_i -1 \geq h_{i+t-1}\geq h_{i+t}$$ and both maps are surjective. Thus both maps are always of full rank for the same reason and $A$ has the SLP. 
\end{proof}

Having established the cases when $k[x,y,z]/(x^a, y^b, z^c, x^\alpha y^\beta)$ has the SLP, we now turn our attention to the cases when it fails the SLP.

\begin{lemma}
\label{lem:3_var_failing}
Fix some integer $c\geq 3$ and let $A=k[x,y,z]/(x^a, y^b, z^c, x^\alpha y^\beta)$. If $B=k[x,y]/(x^a, y^b, x^\alpha y^\beta)$ is not almost centered, then $A$ fails the SLP.
\end{lemma}

\begin{proof}
As always, we may assume that $a+\beta\leq b+\alpha $ after possibly changing the roles of $x$ and $y$. By Lemma \ref{lem:Two_var_Hilbert}, we then have two cases for when $B$ fails to be almost centered. Assume first that $b\geq a+\beta+2$. If we set $i=\alpha $ and $t=a+\beta+c-\alpha -2$, we then claim that $\cdot \ell^t:A_i \to A_{i+t}$ will not have full rank. By Theorem \ref{thm:full_rank_recursion}, it suffices to show that there is two $q$ in the interval from $\max\{0,c-t\}$ to $c-1$ such that $\cdot \overline{\ell}^{2q+t-(c-1)}:B_{i-q}\to B_{i+q+t-(c-1)}$ have full rank for different reasons for those values of $q$. The Hilbert series of $B$ then has a shape as in Figure \ref{fig:All_on_top}. Since $c-t=(2-\beta)-(a-\alpha )\leq 0$ and $c-1\geq 2$, we may look at $q=0$ and $q=2$. Let $h_j=\dim_k(B_j)$. Then $q=0$ gives a map going from a space of dimension $h_\alpha $ to one of dimension $h_{a+\beta-1}$. From $b\geq a+\beta+2$, we then know that $h_\alpha =h_{a+\beta-1}+1$, so the map is only surjective. However, $q=2$ gives a map from a space of dimension $h_{\alpha -2}$ to one of dimension $h_{a+\beta+1}$ where our restriction on $b$ gives that $h_{\alpha -2}=h_{a+\beta+1}-1$. Thus the map is only injective, so $q=0$ and $q=2$ do not have full rank for the same reason and $\cdot \ell^t:A_i \to A_{i+t}$ does not have full rank. 

The other case when $A$ is not almost centered comes from when $a-\alpha \geq 2$, $\beta\geq 2$ and $b\leq a+\beta-2$. Then $B$ has a shape as in Figure \ref{figs:Slanted}. Let $w=(a+\beta-3)-(b+\alpha -a-\beta+2)$. Then $w\geq 1$ since $a+\beta-3\geq \alpha +1$ by $a-\alpha \geq 2$ and $\beta\geq 2$, while 
\[
b+\alpha -a-\beta+2 = \alpha  + (b - (a+\beta-2))\leq \alpha . 
\]
The key property that we will use here to establish that $A$ fails the SLP is that $h_{b+\alpha -a-\beta+2}<h_{a+\beta-3}$ and $h_{b+\alpha -a-\beta}>h_{a+\beta-1}$. Indeed, if we let $i=b+\alpha -a-\beta+2$ and $t=w + c -1$, then $\cdot \ell^t:A_i \to A_{i+t}$ can not have full rank since by Theorem~\ref{thm:full_rank_recursion}, we would need $\cdot \overline{\ell}^{2q+t-(c-1)}:B_{i-q}\to B_{i+q+t-(c-1)}$ to have full rank for the same reason at least for $q=0$ and $q=2$. Note that these are required since $\max\{0,c-t\}=\max\{0,1-w\}=0\leq 2 \leq c-1$. But $q=0$ gives a map that can only be injective since $h_{b+\alpha -a-\beta+2}<h_{a+\beta-3}$, and $q=2$ gives a map that can only be surjective since $h_{b+\alpha -a-\beta}>h_{a+\beta-1}$, giving that $\cdot \ell^t:A_i \to A_{i+t}$ does not have full rank.
\end{proof}

With the three variable case fully understood, we may now prove the main result of this section.

\begin{proof}[Proof of Theorem \ref{thm:SLP_supp_two}]
When $n=1$ there are no monomial almost complete intersections and when $n=2$, we know from \cite[Proposition 4.4]{Codim_2} that any Artinian algebra has the SLP. If $n=3$ and $a_3\leq 2$, the SLP follows by Lemma \ref{lem:3_var_has_SLP}. Next, if $B=k[x_1,x_2]/(x_1^{a_1}, x_2^{a_2}, x_1^{\alpha}x_2^{\beta})$ is almost centered, then
\[
A=k[x_1,x_2]/(x_1^{a_1}, x_2^{a_2}, x_1^{\alpha}x_2^{\beta}) \otimes_k k[x_3,\dots, x_n]/(x_3^{a_3},\dots, x_n^{a_n})
\]
has the SLP since the tensor product of two algebras with the SLP where one is almost centered and the other is symmetric has the SLP by \cite[Theorem 3.5]{Lindsey}. Because the last two conditions are equivalent by Lemma \ref{lem:Two_var_Hilbert}, $A$ has the SLP also in the last case.

For the converse we note that if we have an algebra of the form $C\otimes_k k[x]/(x^d)$ that has the SLP for some $d\geq 1$, then $C$ must have the SLP by \cite[Proposition~5.3]{full_rank_recursion}. Hence the only algebras that we are left to show fails the SLP are those of the form $A=k[x_1, x_2, x_3, x_4]/(x_1^{a_1}, x_2^{a_2}, x_3^2, x_4^2, x_1^{\alpha} x_2^{\beta})$ where $B=k[x_1,x_2]/(x_1^{a_1}, x_2^{a_2}, x_1^{\alpha} x_2^{\beta})$ is not almost centered. Consider a map $\cdot \ell^t:A_i \to A_{i+t}$. If this has full rank, then two applications of Theorem \ref{thm:full_rank_recursion} gives that the maps
\begin{equation}
\begin{split}
\label{eq:maps_suare_recursion}
\cdot \overline{\ell}^{t-2}&:B_{i}\to B_{i+t-2}\\
\cdot \overline{\ell}^{t}&:B_{i-1}\to B_{i+t-1}\\
\cdot \overline{\ell}^{t+2}&:B_{i-2}\to B_{i+t}
\end{split}
\end{equation}
have full rank for the same reason. But when $B$ fails to be almost centered because $a_2\geq a_1 + \beta + 2$, then we know from the proof of Lemma \ref{lem:3_var_failing} that the first and last of the maps in \eqref{eq:maps_suare_recursion} do not have full rank for the same reason when $i=\alpha$ and $t=a_1+\beta-\alpha+1$. Similarly, when $B$ fails to be almost centered because $a_1-\alpha\geq 2$, $\beta\geq 2$ and $a_2\leq a_1+\beta-2$, then the proof of Lemma \ref{lem:3_var_failing} again gives that the first and last map in \eqref{eq:maps_suare_recursion} do not have full rank for the same reason when $i=a_2+\alpha-a_1-\beta+2$ and $t=a_1+\beta-1-i$. Since those are the only cases when $B$ can fail to be almost centered by Lemma \ref{lem:Two_var_Hilbert}, we get that $A$ fails the SLP in these cases and we are done.
\end{proof}

\section{MACIs with symmetric Hilbert series}\label{sect:symmetric_hs}

In this section we will prove that monomial almost complete intersections with symmetric Hilbert series have the SLP. The main tool is the central simple modules of Harima and Watanabe.



To begin with, we will need the following basic result for symmetric polynomials.

\begin{lemma}\label{lem:sym_product}
Let $p,q,r\in k[t]$ be three nonzero polynomials such that $pq=r$. Then two of them are symmetric if and only if all three are symmetric.
\end{lemma}

\begin{proof}
By canceling the largest possible factor of $t$ from both sides of the equality $pq=r$, we may assume that all polynomials have a nonzero constant term. Note that a polynomial $p$ with nonzero constant term is symmetric if and only if 
\[
p(t) = t^{\deg p} p(1/t).
\]
Now, if $p$ and $q$ are symmetric, then
\[
r(t) = p(t)q(t) = t^{\deg p + \deg q} p(1/t) q(1/t) = t^{\deg r} r(1/t),
\]
showing that $r$ is symmetric. If instead $r$ and one of $p$ and $q$ is symmetric, say $p$ is symmetric, then
\[
r(t) = t^{\deg r} r(1/t) = t^{\deg p + \deg q} p(1/t) q(1/t) = p(t) t^{\deg q} q(1/t).
\]
But $r(t)$ also equals $p(t)q(t)$, so 
\[
0=r(t)-r(t) = p(t)\left(q(t) - t^{\deg q} q(1/t)  \right).
\]
Hence we must have that $q(t) - t^{\deg q} q(1/t) = 0$ and that $q$ is symmetric.
\end{proof}

We now give an explicit condition for a monomial almost complete intersection to have a symmetric Hilbert series:

\begin{proposition}\label{prop:symmetry_conditions}
Let $R = k[x_1,\dots,x_n]$ and $I =(x_1^{a_1},\dots,x_n^{a_n}, x_1^{p_1}\cdots x_n^{p_n})$ where $0 \leq p_i < a_i$ for all $i$. Then $R/I$ has a symmetric Hilbert series if and only if there is some integer $N$ such that, up to a possible relabeling of the variables, we have that $p_i=0$ for $i>N$, $p_i\neq 0$ for $i\leq N$, and $$(a_1,a_2,\dots,a_N) = (a_1,a_1+p_2,\dots,a_1+p_2+\dots+p_N).$$ 
\end{proposition}

\begin{proof}
For the first part of this proof, we assume that $N=n$, i.e. $0<p_i<a_i$ for all $i$. We will handle the case when some $p_i = 0$ at the end. 

First, order the variables such that $a_1-p_1\leq a_2-p_2\leq \dots \leq a_n-p_n$. With this ordering, the socle degree of $R/I$ is $p_1+a_2+\cdots + a_n-n$. This is not difficult to see by considering the elements in $R/I$ annihilated by the maximal ideal $(x_1,\dots,x_n)$. The socle of $R/I$ is generated by monomials $x_1^{b_1}x_2^{b_2}\cdots x_n^{b_n}$ where for some $i$, $b_i = p_i$, and for all other indices $b_i = a_i$, i.e.,  $\mathrm{Soc}(R/I) = (x_1^{p_1} x_2^{a_2} \cdots x_n^{a_n}, x_1^{a_1}x_2^{p_2}\cdots x_n^{a_n}, \dots,x_1^{a_1}x_2^{a_2}\cdots x_n^{p_n})$. By our ordering of the variables, the maximal degree among these elements is $p_1+a_2+\cdots + a_n-n$. 

Next, assume that $(a_1,a_2,\dots,a_n) = (a_1,a_1+p_2,\dots,a_1+p_2+\dots+p_n)$. We will prove that $R/I$ is symmetric by induction on $n$. If $n=2$, then again this is contained in Lemma \ref{lem:Two_var_Hilbert}, so we may suppose $n\geq 3$. Using Lemma \ref{lem:Hilbert_add} with $m=x_1^{p_1}\cdots x_{n-1}^{p_{n-1}}$, we have that 
\begin{equation}
\HS_{R/I}(t) = \HS_B(t) + t^{p_1+\cdots + p_{n-1}}\HS_C(t)
\end{equation}
where $B=R/(x_1^{a_1},\dots,x_n^{a_n}, x_1^{p_1}\cdots x_{n-1}^{p_{n-1}})$ and $C=R/(x_1^{a_1-p_1},\dots,x_{n-1}^{a_{n-1}-p_{n-1}}, x_n^{p_n})$. First, by induction we can assume that $A$ is symmetric with reflecting degree given by $r_A = \frac{p_1+a_2+\dots+a_n-n}{2}$. Further, as the complete intersection $B$ is symmetric, $t^{p_1+\cdots + p_{n-1}}\HS_B(t)$ is symmetric with reflecting degree $r_B=\frac{a_1+\dots+a_{n-1}+p_1+\cdots + p_n-n}{2}$. Note that $r_A=r_B$ since $a_n=a_1+p_2+\cdots + p_n$. Hence $\HS_{R/I}(t)$ is a sum of two symmetric polynomials with the same reflecting degree and is therefore also symmetric.

For the converse direction, assume that $\HS_{R/I}(t)$ is symmetric. We will prove that $(a_1,a_2,\dots,a_n) = (a_1,a_1+p_2,\dots,a_1+p_2+\dots+p_n)$ by contradiction. Assume the equality is not true and let $i+1$ be the smallest index for with the formula for $a_{i+1}$ does not hold. Clearly $i>0$. We will the again use Lemma \ref{lem:Hilbert_add}, this time with $m=x_1^{p_1}\cdots x_i^{p_i}$. Then
\[
\HS_{R/I}(t) = \HS_{R/J}(t) + t^{p_1+\cdots + p_i}\HS_{R/K}(t)
\]
where 
\begin{align*}
J&=(x_1^{a_1},\dots,x_n^{a_n}, x_1^{p_1}\cdots x_i^{p_i}), \text{ and}\\
K&=(x_1^{a_1-p_1},\dots,x_i^{a_i-p_i}, x_{i+1}^{a_{i+1}},\dots, x_n^{a_n}, x_{i+1}^{p_{i+1}}\cdots x_n^{p_n}).
\end{align*}
Note that $R/J$ has socle degree $s=p_1+a_2+\cdots + a_n-n$ and is symmetric since it is a tensor product of a complete intersection and $k[x_1,\dots, x_i]/(x_1^{a_1},\dots,x_i^{a_i}, x_1^{p_1}\cdots x_i^{p_i})$, which is symmetric by assumption on $a_1,\dots, a_i$ and the earlier proven direction. Hence $\HS_{R/I}(t)$ and $\HS_{R/J}(t)$ are both symmetric polynomials of the same degree $s$ having non-zero constant term. This forces $t^{p_1+\cdots + p_i}\HS_{R/K}(t)$ to also be symmetric with reflecting degree $s/2$. Indeed, an argument similar to that of Lemma \ref{lem:sym_product} gives that if $r(t)=t^{p_1+\cdots + p_i}\HS_{R/K}(t)$, then $t^sr(1/t)=r(t)$.
Since $r$ is supported from $p_1+\cdots+p_i$ to $a_1+\cdots + a_i+p_{i+1}+a_{i+2}+\cdots + a_n-n$ (where we again used the formula for the socle degree of our almost monomial complete intersections), $t^sr(1/t)=r(t)$ gives that
\[
(p_1+\cdots + p_i) + a_1+\cdots + a_i+p_{i+1}+a_{i+2}+\cdots + a_n-n = s = p_1+a_2+\cdots + a_n-n.
\]
But this is equivalent to $a_1+p_2+\cdots + p_{i+1}=a_{i+1}$, contradicting our choice of $i$. Hence we must have that $(a_1,a_2,\dots,a_n) = (a_1,a_1+p_2,\dots,a_1+p_2+\dots+p_n)$, as desired.

We now consider the case in which some $p_i=0$ for $i>N$. Define $B = k[x_1,\dots, x_N]/(x_1^{a_1},\dots,x_N^{a_N}, x_1^{p_1}\cdots x_N^{p_N})$ and $C = k[x_{N+1},\dots x_n]/(x_{N+1}^{a_{N+1}},\dots, x_n^{a_n})$. Then $R/I \cong B \otimes_k C$, and 
\[
\HS_{R/I}(t) = \HS_B(t)\cdot \HS_C(t).
\] 
Note that $C$ is a symmetric polynomial, so Lemma \ref{lem:sym_product} gives that $\HS_{R/I}(t)$ is symmetric if and only if $\HS_B(t)$ is symmetric. But $B$ is of the form where all $p_1,\dots, p_m\neq 0$, so we can use the earlier discussion on $B$. 
\end{proof}

\begin{theorem}\label{thm:slp_symmetric_aci}
    Let $R = k[x_1,\dots,x_n]$ and $I =(x_1^{a_1},\dots,x_n^{a_n}, x_1^{p_1}\cdots x_n^{p_n})$ with $0 \leq p_i < a_i$ for all $i$. If the Hilbert series of $A = R/I$ is symmetric, then $A$ has the SLP. 
\end{theorem}
\begin{proof}
We start with the assumption that $(a_1,a_2,\dots,a_n) = (a_1,a_1+p_2,\dots,a_1+p_2+\dots+p_n)$, using a similar argument as in the previous proposition to reduce to the case where all $p_i$ are nonzero. We will prove the theorem using central simple modules.

We have, with respect to the linear form $x_n$, two central simple modules. The first is 
\begin{align*}
V_1 &= \frac{(0:x_n^{a_n})+(x_n)}{(0:x_n^{a_n-1})+(x_n)} \\
&= \frac{A}{( x_1^{p_1} \cdots x_{n-1}^{p_{n-1}}, x_n)} \\
&\cong \frac{R}{(x_1^{a_1}, \dots, x_{n-1}^{a_{n-1}}, x_n, x_1^{p_1} \cdots x_{n-1}^{p_{n-1}})}, 
\end{align*}
so $f_1 = a_n$ and hence
$\widetilde{V}_1 \cong R/(x_1^{a_1}, \dots, x_n^{a_n}, x_1^{p_1} \cdots x_{n-1}^{p_{n-1}})$. The second central simple module is
\begin{align*}   
 V_2 &= \frac{(0:x_n^{p_n})+(x_n)}{(0:x_n^{p_n-1})+(x_n)} \\
 &= \frac{(x_1^{p_1} \cdots x_{n-1}^{p_{n-1}}, x_n)}{(x_n)} \\ &\cong R/(x_1^{a_1-p_1},\dots,x_{n-1}^{a_{n-1}-p_{n-1}},x_n)[-(p_1+\dots+p_{n-1})].
 \end{align*}
 with $f_2 = p_n$ and
 $\widetilde{V}_2 \cong R/(x_1^{a_1-p_1},\dots,x_{n-1}^{a_{n-1}-p_{n-1}}, x_{n}^{p_{n}})$ (with the same shifted grading as above).

Notice that $V_2$ always has the SLP and a symmetric Hilbert series, as it is a monomial complete intersection. Furthermore, if $\HS_{R/I}(t)$ is symmetric, we have $(a_1,\dots,a_n) =(a_1,a_1+p_2,\dots,a_1+p_2+\dots+p_n)$ by Proposition \ref{prop:symmetry_conditions}. So $\HS_{V_1}(t)$ is symmetric as well since $V_1$ is a monomial almost complete intersection in one fewer variable than $R/I$ and $(a_1,\dots,a_{n-1}) =(a_1,a_1+p_2,\dots,a_1+p_2+\dots+p_{n-1})$. 

Thus we may apply Theorem \ref{hw-csm-thm} here, noticing that the conditions for $\HS_{R/I}(t)$ to be symmetric are the same as the conditions for the reflecting degrees to coincide as needed, as we now explain. The reflecting degree of $\widetilde{V}_1$ is $r_1 = \frac{p_1+a_2+\dots+a_n-n}{2}$. Observe that the socle degree of $R/I$ is exactly the same as that of $\widetilde{V}_1$, hence the reflecting degree of $R/I$ is $r = r_1$. Since $\HS_{V_2}(t)$, and therefore also $\HS_{\widetilde{V}_2}(t)$, are symmetric, and we must have $r_2 = r = r_1$ if $\HS_{R/I}(t)$ is to be symmetric by using that $\HS_{R/I}(t) = \HS_{\widetilde{V}_1}(t) + \HS_{\widetilde{V}_2}(t)$. So, from Theorem \ref{hw-csm-thm}, $R/I$ has the SLP as long as $V_2$ has the SLP. This is always the case, following from an induction on $n$, with the base case being that when $n = 2$, any algebra and thus any monomial almost complete intersection has the SLP.  
\end{proof}

We conclude our discussion of monomial algebras by giving an example illustrating the results in this section. The below is the smallest (in terms of the degrees of the generators) nontrivial MACI in $n$ variables, where the non-pure-power generator has full support, which has a symmetric Hilbert series. 

\begin{example}
Let $R = k[x_1,\dots,x_n]$ and $I = (x_1^2,x_2^3,\dots,x_n^{n+1}, x_1 \cdots x_n$). Then $R/I$ has a symmetric Hilbert series since the exponents satisfy the conditions of Proposition \ref{prop:symmetry_conditions}. For example, when $n = 4$  we have $\HS_{R/I}(t) = 1+4t+9t^2+15t^3+19t^4+19t^5+15t^6+9t^7+4t^8+t^9$. So, by the previous theorem $R/I$ has the strong Lefschetz property. 

\end{example}

\begin{remark}
From the work of Wiebe \cite{wiebe}, we know that a graded ideal $I$ inherits the strong (or weak) Lefschetz property from its initial ideal. As an immediate corollary of the above theorem, we observe that any $k$-algebra with a symmetric Hilbert series defined by an ideal $I$ with an initial ideal that is a monomial complete intersection has the strong Lefschetz property. This holds in particular for Gorenstein algebras, since they have symmetric Hilbert series. Since almost complete intersections can never be Gorenstein \cite{ACI_not_Gorenstein}, and since an ideal cannot have more minimal generators than its initial ideal, we know that such a Gorenstein algebra must be a complete intersection. However, it is possible no such complete intersection exists - experiments in Macaulay2 (\cite{M2}) have not yielded any concrete examples. 
\end{remark}

\section*{Acknowledgments}
Work on this project began at the conference ``Lefschetz properties in algebra, geometry, topology and combinatorics", held in Kraków, Poland, in June 2024. We thank the organizers for their support and hospitality.

\bibliographystyle{plain}
\bibliography{references}

\end{document}